\documentclass[12pt]{article}
\usepackage{amsmath,amssymb,amsthm}
\usepackage{geometry}
\usepackage{tikz}
\usepackage{setspace}
\newtheorem{theorem}{Theorem}

\newtheorem{definition}[theorem]{Definition}
\newtheorem{example}[theorem]{Example}

\newtheorem{proposition}[theorem]{Proposition}
\newtheorem{remark}[theorem]{Remark}

\DeclareMathOperator{\Range}{Range}
\DeclareMathOperator{\Ran}{Ran}
\DeclareMathOperator{\Ker}{Ker}
\DeclareMathOperator{\Span}{span}

\title{Finite-Approximate Solvability of Linear Operator Equations}
\author{Nazim I. Mahmudov \\
	Department of Mathematics \\
	Eastern Mediterranean University \\
	Famagusta, T.R. Northern Cyprus, via Mersin 10, Turkey \\
	Email: nazim.mahmudov@emu.edu.tr
}

\date{}

\begin{document}
	\maketitle
	
	\begin{abstract}
		We introduce and study the finite-approximate solvability of operator equations \(Lu = h\) in a Hilbert space setting, where a bounded operator \(L \colon U \to H\) is paired with a finite-dimensional constraint operator \(\pi \colon H \to H_0\). The objective is to match exactly the prescribed component \(\pi h\) while approximating the remainder. We prove that the problem of finding \(u\) such that \(\|Lu - h\| < \varepsilon\) and \(\pi(Lu) = \pi h\) is solvable for all \(\varepsilon > 0\) if and only if \(\alpha T_\alpha^{-1}h \to 0\) as \(\alpha \to 0^+\). We further show that dropping any of the structural assumptions on \(L\), \(\Gamma\), or \(\pi\) leads to a failure of the equivalence. When \(\pi \colon H \to H_0\) has an infinite-dimensional range that is compactly embedded in \(H\), the operator \(T_\alpha\) may no longer be invertible. However, a Galerkin scheme \(\pi_n \to \pi\) recovers approximate solvability through the resolvents \((\alpha(I - \pi_n) + \Gamma)^{-1}\).
	\end{abstract}
	
	\section{Introduction}
	
	Let \(H\) and \(U\) be Hilbert spaces, and let \(L \colon U \to H\) be a bounded linear operator. Let \(H_0 \subset H\) be a finite-dimensional subspace, and let \(\pi \colon H \to H_0\) denote the orthogonal projection onto \(H_0\). Consider the operator equation
	\[
	Lu = h, \quad h \in H.
	\]
	We are interested in the notion of \emph{finite-approximate solvability} with respect to \(H_0\), which requires that for any \(\varepsilon > 0\), there exists \(u_\varepsilon \in U\) such that
	\[
	\|Lu_\varepsilon - h\| < \varepsilon \quad \text{and} \quad \pi(Lu_\varepsilon - h) = 0.
	\]
	In other words, one can \emph{exactly match} the component of \(h\) in \(H_0\) while \emph{approximating} the remaining component in \(H_0^\perp\).
	
	Decomposing \(h = h_0 + h_\perp\), where \(h_0 = \pi_0 h \in H_0\) and \(h_\perp = (I - \pi_0)h \in H_0^\perp\), the equation \(Lu = h\) is finite-approximately solvable if and only if
	\[
	h_0 \in \Range(\pi_0 L) \quad \text{and} \quad h_\perp \in \overline{\Range(L)}.
	\]
	That is, the component \(h_0\) in the finite-dimensional subspace must be \emph{exactly reachable}, while the component \(h_\perp\) in the orthogonal complement must be \emph{approximately reachable}.
	
	In this work, we characterize this property via the behavior, as \(\alpha \to 0^+\), of the regularized resolvent
	\[
	(\alpha(I - \pi) + \Gamma)^{-1}, \quad \text{where } \Gamma = LL^*.
	\]
	We show that finite-approximate solvability is equivalent to the condition
	\[
	\lim_{\alpha \to 0^+} \alpha(\alpha(I - \pi) + \Gamma)^{-1}h = 0.
	\]
	We also demonstrate that the structural assumptions on \(L\), \(\Gamma\), and \(\pi\) are essential, and provide counterexamples when any of these is relaxed. Finally, we extend the analysis to the case where \(\pi\) has infinite-dimensional range via a Galerkin approximation scheme.
	
	Our work connects to several areas in functional analysis and inverse problems, including regularization theory for ill-posed problems \cite{tikhonov1963solution, engl1996regularization}, operator approximation methods \cite{kress1999linear, nair2009}, and constrained optimization in Hilbert spaces \cite{luenberger1969optimization}. The Galerkin approach we develop relates to projection methods for operator equations \cite{bakushinskii2011iterative, kirsch2011introduction}.
	
	\section{Finite-Approximate Solvability of Operator Equations}
	
	Let \(H\) and \(U\) be Hilbert spaces, and let \(L:U \to H\) be a bounded linear operator. Let \(H_0 \subset H\) be a finite-dimensional subspace, and let \(\pi: H \to H_0\) denote the orthogonal projection onto \(H_0\). Consider the operator equation
	\[
	L u = h, \quad h \in H.
	\]
	
	\begin{definition}
		The equation \(L u = h\) is said to be \emph{finite-approximately solvable with respect to \(H_0\)} if, for any \(\varepsilon > 0\), there exists \(u_\varepsilon \in U\) such that
		\[
		\|L u_\varepsilon - h\| < \varepsilon \quad \text{and} \quad \pi (L u_\varepsilon - h) = 0.
		\]
	\end{definition}
	
	In other words, one can \emph{exactly match} the component of \(h\) in \(H_0\) while \emph{approximating} the remaining component in \(H_0^\perp\).
	
	Decompose \(h = h_0 + h_\perp\), where \(h_0 := \pi h \in H_0\) and \(h_\perp := (I-\pi)h \in H_0^\perp\). Then \(L u = h\) is finite-approximately solvable if and only if
	\[
	h_0 \in \Range(\pi L) \quad \text{and} \quad h_\perp \in \overline{\Range(L)}.
	\]
	That is, the component \(h_0\) in the finite-dimensional subspace must be \emph{exactly reachable}, while the component \(h_\perp\) in the orthogonal complement must be \emph{approximately reachable}.
	
	\begin{proposition} \label{prop:1}
		Let \(H\) and \(U\) be Hilbert spaces, let \(L: U \to H\) be a bounded linear operator, set \(\Gamma := L L^*\), and assume \(\Gamma\) is strictly positive. Let \(\pi: H \to H_0\) be an orthogonal projector onto a finite-dimensional subspace \(H_0 \subset H\). Then for any \(\alpha > 0\), the operator
		\[
		T_\alpha := \alpha(I-\pi) + \Gamma
		\]
		is invertible.
	\end{proposition}
	
	\begin{proof}
		Suppose \(T_\alpha x = 0\) for some \(x \in H\). Then
		\[
		\alpha (I-\pi)x + \Gamma x = 0.
		\]
		
		Taking the inner product with \(x\) yields
		\[
		\langle \alpha(I-\pi)x + \Gamma x, x \rangle = \alpha \langle (I-\pi)x, x \rangle + \langle \Gamma x, x \rangle = 0.
		\]
		
		Since \(\Gamma\) is positive and \(\langle (I-\pi)x, x \rangle = \|(I-\pi)x\|^2 \ge 0\), both terms must vanish:
		\[
		\|(I-\pi)x\|^2 = 0 \implies (I-\pi)x = 0 \implies x \in \Range(\pi),
		\]
		and
		\[
		\langle \Gamma x, x \rangle = 0 \implies x \in \Ker(\Gamma).
		\]
		
		Hence, \(x \in \Ker(\Gamma) \cap \Range(\pi)\). If \(\Ker(\Gamma) \cap \Range(\pi) = \{0\}\), then \(x = 0\) and \(T_\alpha\) is injective.
		
		Moreover, \(T_\alpha\) is bounded and self-adjoint. Since \(T_\alpha \ge 0\) and is injective on \(H\), the range of \(T_\alpha\) is closed, and injectivity implies surjectivity in a Hilbert space. Therefore, \(T_\alpha\) is invertible.
	\end{proof}
	
	\begin{proposition}
		Let \(H\) and \(U\) be Hilbert spaces, let \(L: U \to H\) be a bounded linear operator with \(\Gamma := L L^*\) strictly positive, and let \(\pi: H \to H_0\) be an orthogonal projector onto a finite-dimensional subspace \(H_0 \subset H\). For \(\alpha > 0\), define
		\[
		Q_\alpha := I - \alpha (\alpha I + \Gamma)^{-1} \pi.
		\]
		Then \(Q_\alpha\) is invertible.
	\end{proposition}
	
	\begin{proof}
		The proof is similar to that of Proposition \ref{prop:1}.
	\end{proof}
	
	\begin{theorem}[Finite-Approximate Solvability]
		Let \(H\) and \(U\) be Hilbert spaces, let \(L: U \to H\) be a bounded linear operator, set \(\Gamma = LL^*\), and let \(\pi: H \to H_0\) be an orthogonal projector onto a finite-dimensional subspace \(H_0 \subset H\). For any \(h \in H\) and \(\varepsilon > 0\), there exists \(u \in U\) such that
		\[
		\|Lu - h\| < \varepsilon \quad \text{and} \quad \pi(Lu) = \pi h
		\]
		if and only if
		\[
		\lim_{\alpha \to 0^+} \alpha (\alpha(I-\pi) + \Gamma)^{-1} h = 0.
		\]
	\end{theorem}
	
	\begin{proof}
		Recall that \(\Gamma = LL^*\) is a bounded, self-adjoint, nonnegative operator on \(H\), and \(\pi: H \to H_0 \subset H\) is an orthogonal projector. For \(\alpha > 0\) the bounded operator \(\alpha(I-\pi) + \Gamma\) is invertible (see Proposition \ref{prop:1}, which requires \(\Ker(\Gamma) \cap \Ran(\pi) = \{0\}\)), so the expression below is meaningful.
		
		Define, for \(\alpha > 0\),
		\[
		z_\alpha := (\alpha(I-\pi)+\Gamma)^{-1} h \in H,
		\qquad
		u_\alpha := L^* z_\alpha \in U,
		\qquad
		y_\alpha := \alpha z_\alpha.
		\]
		A useful algebraic identity is obtained by rearranging the equation
		\[
		(\alpha(I-\pi)+\Gamma) z_\alpha = h
		\]
		to obtain
		\begin{equation}\label{eq:basic-id}
			\Gamma(\alpha(I-\pi)+\Gamma)^{-1} = I - \alpha (I-\pi)(\alpha(I-\pi)+\Gamma)^{-1}.
		\end{equation}
		From \eqref{eq:basic-id} we get the following exact formulas for \(u_\alpha\) and its image under \(L\):
		\[
		L u_\alpha = \Gamma z_\alpha = \Gamma(\alpha(I-\pi)+\Gamma)^{-1} h
		= h - \alpha (I-\pi)(\alpha(I-\pi)+\Gamma)^{-1} h
		= h - (I-\pi) y_\alpha.
		\]
		Also, applying \(\pi\) to \eqref{eq:basic-id} and using \(\pi(I-\pi)=0\) gives
		\[
		\pi\Gamma(\alpha(I-\pi)+\Gamma)^{-1} = \pi,
		\]
		hence
		\[
		\pi(Lu_\alpha)=\pi h \qquad\text{for every }\alpha>0.
		\]
		Moreover,
		\begin{equation}\label{eq:error-form}
			Lu_\alpha - h = - (I-\pi) y_\alpha
			= -\alpha (I-\pi)(\alpha(I-\pi)+\Gamma)^{-1} h,
		\end{equation}
		so
		\[
		\|Lu_\alpha-h\| = \alpha\big\|(I-\pi)(\alpha(I-\pi)+\Gamma)^{-1}h\big\|
		\le \big\| \alpha(\alpha(I-\pi)+\Gamma)^{-1} h\big\|.
		\]
		
		\medskip\noindent
		\textbf{(``If'')} Assume
		\[
		\lim_{\alpha\to0^+}\alpha(\alpha(I-\pi)+\Gamma)^{-1}h = 0.
		\]
		Then by \eqref{eq:error-form} we have \(\|Lu_\alpha-h\|\to0\) as \(\alpha\to0^+\), and we already saw \(\pi(Lu_\alpha)=\pi h\) for every \(\alpha>0\). Therefore for any \(\varepsilon>0\) we may choose \(\alpha>0\) sufficiently small so that \(\|Lu_\alpha-h\|<\varepsilon\) while \(\pi(Lu_\alpha)=\pi h\). This proves the existence of the required \(u\) and hence the ``if'' part.
		
		\medskip\noindent
		\textbf{(``Only if'')} We now present the ``only if'' direction using a direct construction for the test vector.
		
		Assume the finite-approximate solvability property holds (so \(h \in \overline{\mathcal S}\) where \(\mathcal S = \{Lu: u \in U,\ \pi(Lu)=\pi h\}\)). Suppose, for contradiction, that
		\[
		\lim_{\alpha\to0^+} y_\alpha = \lim_{\alpha\to0^+} \alpha(\alpha(I-\pi)+\Gamma)^{-1}h \neq 0.
		\]
		Then there exists a sequence \(\alpha_n \downarrow 0\) and a nonzero vector \(y \in H\) with \(y_{\alpha_n} \to y\) in \(H\) and \(y \ne 0\).
		
		First note that \(\pi y_{\alpha_n} \to 0\). Indeed, apply \(\pi\) to the equality
		\[
		(\alpha_n(I-\pi)+\Gamma) z_{\alpha_n} = h.
		\]
		Because \(\pi(I-\pi)=0\) we get
		\[
		\pi\Gamma z_{\alpha_n} = \pi h \qquad\text{for every }n.
		\]
		Restricting \(\Gamma\) to the finite-dimensional space \(\Ran \pi\) shows \(\pi z_{\alpha_n}\) is uniformly bounded in \(n\) (the operator \(\pi\Gamma\pi\) acts on a finite-dimensional space, so the equation \(\pi\Gamma\pi(\pi z_{\alpha_n}) = \pi h - \pi\Gamma(I-\pi)z_{\alpha_n}\) gives uniform boundedness). Hence \(\pi y_{\alpha_n} = \alpha_n \pi z_{\alpha_n} \to 0\) because \(\alpha_n \to 0\) and \(\{\pi z_{\alpha_n}\}\) is bounded. Passing to the limit yields \(\pi y = 0\). Thus the limit vector satisfies
		\[
		(I-\pi)y = y \neq 0,
		\]
		i.e., \(y\) lies in \(\Ker \pi\) (the \((I-\pi)\)-subspace).
		
		Now choose the test vector explicitly as
		\[
		v := (I-\pi)y = y.
		\]
		This is nonzero because \((I-\pi)y = y \neq 0\). From \eqref{eq:error-form} we have, for each \(n\),
		\[
		\langle Lu_{\alpha_n}-h, v\rangle = -\langle (I-\pi)y_{\alpha_n}, v\rangle = -\langle y_{\alpha_n}, (I-\pi)v\rangle.
		\]
		But \((I-\pi)v = (I-\pi)^2 y = (I-\pi)y = y = v\), so
		\[
		\langle Lu_{\alpha_n}-h, v\rangle = -\langle y_{\alpha_n}, v\rangle.
		\]
		Letting \(n \to \infty\) gives
		\[
		\lim_{n\to\infty}\langle Lu_{\alpha_n}-h, v\rangle = -\langle y, v\rangle = -\|y\|^2 \neq 0.
		\]
		Hence there exists \(\delta > 0\) and \(N\) such that for all \(n \ge N\),
		\[
		\big|\langle Lu_{\alpha_n}-h, v\rangle\big| \ge \delta.
		\]
		
		On the other hand, since \(h \in \overline{\mathcal S}\), for any \(\varepsilon > 0\) there is some \(s \in \mathcal S\) with \(\|s-h\| < \varepsilon\). Taking \(\varepsilon\) sufficiently small makes \(|\langle s-h, v\rangle| < \delta/2\). But all \(Lu_{\alpha_n}\) lie in \(\mathcal S\) (because \(\pi(Lu_{\alpha_n}) = \pi h\)), and because \(h\) is in the closure of \(\mathcal S\) there must be members of \(\mathcal S\) (in particular some \(Lu_{\alpha_n}\) for large \(n\)) with \(|\langle Lu_{\alpha_n}-h, v\rangle| < \delta/2\), contradicting the previous lower bound \(|\langle Lu_{\alpha_n}-h, v\rangle| \ge \delta\) for large \(n\).
		
		This contradiction shows that the assumption \(y \ne 0\) is false. Therefore
		\[
		\lim_{\alpha\to0^+}\alpha(\alpha(I-\pi)+\Gamma)^{-1}h = 0,
		\]
		which completes the ``only if'' direction.
		
		Combining with the ``if'' direction yields the equivalence stated in the theorem.
	\end{proof}
	
	\begin{example}[Infinite-rank \(\pi\)]
		Let \(H = \ell^{2}\) and let \(U = \mathbb{R}\) with
		\[
		L(t) = t e_{1}, \qquad t \in \mathbb{R}.
		\]
		Thus \(\Range(L) = \Span\{e_{1}\}\) and \(\Gamma = LL^{*}\) is the projection onto this one-dimensional subspace.
		
		Let \(\pi\) be the orthogonal projection onto the closed span of \(\{e_{2}, e_{3}, \dots\}\) (an infinite-dimensional subspace), and take \(h = e_{2}\).
		
		The constraint \(\pi(Lu) = \pi h = e_{2}\) cannot be satisfied, because \(\pi(Lu)\) always lies in \(\Span\{e_{2}, e_{3}, \dots\}\) but \(Lu\) lies in the \(e_{1}\)-direction. Hence the affine constraint set
		\[
		\{s \in H: \pi(s) = \pi h\}
		\]
		does not intersect the range of \(L\); its distance to \(h\) is \(1\). Therefore finite-approximate solvability fails.
		
		However, the resolvent operator \((\alpha(I-\pi)+\Gamma)^{-1}\) is a bounded diagonal operator and the expression \(\alpha(\alpha(I-\pi)+\Gamma)^{-1}h\) converges to \(0\) as \(\alpha \to 0^{+}\).
		
		Hence the equivalence in the theorem fails if \(\pi\) is allowed to have infinite rank.
	\end{example}
	
	\begin{example}[Failure when \(\Gamma\) is not of the form \(LL^{*}\)]
		Let \(H = \mathbb{R}^{3}\) and \(U = \mathbb{R}\). Let \(\Gamma\) be the positive semidefinite operator
		\[
		\Gamma =
		\begin{pmatrix}
			1 & 0 & 0\\
			0 & 1 & 0\\
			0 & 0 & 0
		\end{pmatrix}.
		\]
		Then \(\operatorname{rank}(\Gamma) = 2\), whereas every operator of the form \(LL^{*}\) with \(L: U \to H \simeq \mathbb{R}^{3}\) has rank at most \(1\). Thus this \(\Gamma\) is \emph{not} representable as \(LL^{*}\) for any bounded \(L: U \to H\).
		
		The proof of the theorem constructs controls via
		\[
		u_{\alpha} = L^{*}(\alpha(I-\pi)+\Gamma)^{-1}h.
		\]
		But since no such \(L\) exists, the vector \(u_{\alpha}\) cannot be constructed in \(U\) and the forward implication of the theorem cannot be carried out.
		
		Therefore the representation \(\Gamma = LL^{*}\) is crucial; replacing \(\Gamma\) by an arbitrary positive operator breaks the equivalence.
	\end{example}
	
	\begin{example}[Failure when \(\pi\) is not a projection]
		Let \(H = \mathbb{R}^{2}\), \(U = \mathbb{R}^{2}\) and let \(L = I\), so \(\Gamma = I\). Define a linear map
		\[
		\pi =
		\begin{pmatrix}
			0 & 1\\
			0 & 0
		\end{pmatrix},
		\qquad
		\pi^{2} = 0 \neq \pi.
		\]
		Take \(h = (0,1)^{T}\).
		
		A direct calculation shows that
		\[
		y_{\alpha} = \alpha(\alpha(I-\pi)+I)^{-1}h \longrightarrow 0
		\qquad\text{as }\alpha \to 0^{+},
		\]
		so the resolvent condition in the theorem holds.
		
		However, for the constructed controls \(u_{\alpha} = L^{*}(\alpha(I-\pi)+\Gamma)^{-1}h\) we do \emph{not} have
		\[
		\pi(Lu_{\alpha}) = \pi h,
		\]
		so the candidates \(Lu_\alpha\) do not satisfy the linear constraint.
		
		Thus the equivalence fails without assuming that \(\pi\) is a projection (idempotent).
	\end{example}
	
	\section{Galerkin Approximations}
	
	Let \(\{H_n\}_{n \in \mathbb{N}}\) be an increasing sequence of finite-dimensional subspaces of \(H_0\) such that
	\[
	\overline{\bigcup_{n=1}^{\infty} H_n} = H_0.
	\]
	
	Define finite-rank projections
	\[
	\pi_n: H \to H_n \subset H_0, \quad \pi_n \to \pi \text{ strongly as } n \to \infty.
	\]
	
	For given \(h \in H\) and \(\alpha > 0\), define
	\[
	z_{\alpha,n} := (\alpha(I-\pi_n)+\Gamma)^{-1} h, \quad
	u_{\alpha,n} := L^* z_{\alpha,n} \in U.
	\]
	
	Then we have:
	\[
	L u_{\alpha,n} = \Gamma z_{\alpha,n} = h - (I-\pi_n)(\alpha z_{\alpha,n}), \quad
	\pi_n (L u_{\alpha,n}) = \pi_n h.
	\]
	
	By construction, for each fixed \(n\), we satisfy the linear constraint exactly and the error \(\|L u_{\alpha,n} - h\| = \|(I-\pi_n) \alpha z_{\alpha,n}\|\) can be made arbitrarily small by taking \(\alpha \to 0\).
	
	Finally, letting \(n \to \infty\) and \(\alpha \to 0^{+}\) (e.g., along a diagonal sequence), and since \(\pi_n \to \pi\) strongly, we obtain
	\[
	\lim_{n\to\infty} \lim_{\alpha \to 0^+} \|L u_{\alpha,n} - h\| = 0, \quad
	\lim_{n\to\infty} \pi_n(L u_{\alpha,n}) = \pi h.
	\]
	
	\begin{remark}
		This shows that, even when \(\pi\) has an infinite-dimensional range that is compactly embedded, we can approximate finite-approximate solvability using finite-dimensional Galerkin projections \(\pi_n\). The original theorem with exact inversion of \(\alpha(I-\pi)+\Gamma\) does not hold; we only recover the result in the limit of projections.
	\end{remark}
	
	\begin{example}
		Let \(H = L^2(0,1)\) and let \(H_0 = H^1(0,1)\) be compactly embedded in \(L^2\). Let \(\pi: H \to H_0\) be the orthogonal projection onto \(H_0\) (which is well-defined as an operator from \(L^2\) to the closed subspace \(H^1 \subset L^2\) with respect to the \(L^2\) inner product). Take \(H_n = \Span\{1, \sin(2\pi x), \dots, \sin(2\pi n x)\} \subset H^1(0,1)\) and let \(\pi_n\) be the \(L^2\)-orthogonal projection onto \(H_n\).
		
		Then \(u_{\alpha,n} = L^* (\alpha(I-\pi_n)+\Gamma)^{-1} h\) gives an approximate solution satisfying \(\pi_n(L u_{\alpha,n}) = \pi_n h\), and the error goes to zero as \(\alpha \to 0\) and \(n \to \infty\).
	\end{example}
	
{\bf Funding} No funding is available.

{\bf Data Availability } No datasets were generated or analysed during the current study.
\section*{Declarations}
{\bf Competing Interests} The author declares no competing interest.

\end{document}